\documentclass[11pt,a4paper,reqno]{article}
\author{Mara Ungureanu}
\title{Refined de Jonqui\`eres divisors and secant varieties \\ on algebraic curves}
\date{}

\AtEndDocument{\medskip
\noindent\begin{footnotesize}\textsc{Albert-Ludwigs-Universit\"at Freiburg, Mathematisches Institut, Abteilung Reine Mathematik, Ernst-Zermelo-Str 1, 79104 Freiburg}
\smallskip\\
\textit{E-mail address:} \texttt{mara.ungureanu@math.uni-freiburg.de}
\end{footnotesize}}

\usepackage[style=numeric]{biblatex}
\usepackage{amsmath, amssymb, amsthm,mathrsfs}
\usepackage{enumitem}
\usepackage{color}
\usepackage{mathtools}
\usepackage[a4paper,left=3cm,right=3cm,bottom=4cm]{geometry}
\usepackage{lmodern}
\usepackage{indentfirst}
\usepackage{tikz}
\usetikzlibrary{matrix,arrows}
\usepackage{epstopdf}
\usepackage{float}
\usepackage{microtype}

\usepackage[utf8]{inputenc}
\usepackage[T1]{fontenc}

\newcommand{\grd}{g^r_d}

\newcommand{\grdop}[2]{g^{#1}_{#2}}
\DeclareMathOperator{\im}{\text{im}}

\DeclareMathOperator{\oo}{\mathcal{O}}

\DeclareMathOperator{\p}{\mathbb{P}}
\DeclareMathOperator{\pic}{\text{Pic}}

\newcommand{\Grd}{G^r_d(C)}
\newcommand{\Wrd}{W^r_d(C)}
\newcommand{\Crd}{C^r_d}

\newtheorem{thm}{Theorem}[section] 
\newtheorem{lemma}[thm]{Lemma}
\newtheorem{prop}[thm]{Proposition}
\newtheorem{cor}[thm]{Corollary}
\newtheorem*{hypo}{Induction hypothesis}

\theoremstyle{definition}
\newtheorem{defi}{Definition}[section]
\theoremstyle{definition}

\theoremstyle{remark}
\newtheorem{rem}{Remark}[section]

\begin{document}

\maketitle

\begin{abstract}
The aim of this paper is to provide another perspective on secant varieties on algebraic curves by reformulating the problem in terms of refined de Jonqui\`eres divisors, that is divisors on the curve with prescribed multiplicities and dimensions of their spaces of global sections.  We are able to both recover some already known results and to obtain some new statements concerning the dimension theory of secant varieties.  We do this via the study of the dimension theory of refined de Jonqui\`eres divisors in some relevant cases and degeneration arguments.
\end{abstract}

\section{Introduction}
Secant varieties to algebraic curves are some of the most studied objects in classical algebraic geometry.  Their enumerative geometry has been of particular interest, with Castelnuovo \autocite{Cas}, Cayley \autocite{Cay}, and later MacDonald \autocite{ACGH} providing various formulas predicting the number of secant $k$-planes to a curve embedded in a projective space of some fixed dimension.  

Despite their long history, it was only quite recently that the validity of these enumerative formulas has been verified in most cases, independently by Cotterill \autocite{Co} and Farkas \autocite{Fa2}.   More precisely, for a general curve $C$ of genus $g$ equipped with a linear series $l$ of degree $d$ and dimension $r$, denote by $V_e^{e-f}(l)$ the variety of effective divisors of degree $e$ on $C$ that impose at most $e-f$ conditions on $l$.  In other words, a divisor $D$ belonging to $V_e^{e-f}(l)$ has the property that the series $l-D$ has degree $d-e$ and dimension at least $r-e+f$.  The space $V_e^{e-f}(l)$  has the structure of a determinantal subvariety of the $e$-th symmetric product $C_e$ of $C$ and as such it has an expected dimension 
\[ \text{exp}\dim V_e^{e-f}(l) = e-f(r+1-e+f).\]
In \autocite[Theorem 0.1]{Fa2} it is shown that $V_e^{e-f}(l)$ is empty for all $l$ of degree $d$ and dimension $r$ if \[\text{exp}\dim V_e^{e-f}(l)<-\rho(g,r,d),\] where $\rho(g,r,d):=g-(r+1)(g-d+r)$ is the Brill-Noether number of $l$.  As a consequence of this result, we have (cf.~\autocite[Corollary 0.3]{Fa2}) that if the variety $V_e^{e-f}(l)$ is non-empty, then it is equidimensional and of expected dimension $e-f(r+1-e+f)\geq 0$.

The issue of non-emptiness of secant varieties when the expected dimension is non-negative is unfortunately not as easily established. The cycle class of $V_e^{e-f}(l)$ computed by MacDonald displays positivity features only in the case of a non-special linear series $l$, however new results by Cotterill, He, and Zhang \autocite{CHZ} indicate that formulas with better positivity properties can be found in some cases.  An older result of Coppens and Martens \autocite[Theorem 1.2]{CM} already addresses the non-emptiness in a different context and states that if $d\geq 2e-1$ and 
\begin{equation}\label{eq:cmbound}
e-f(r+1-e+f)\geq r-e+f,
\end{equation}
then $V_e^{e-f}(l)$ is not empty, for any curve $C$ and linear series $l$.  Moreover \autocite[Theorem 0.5]{Fa2} also provides non-emptiness results for certain ranges of the parameters.

The purpose of this article is to offer another perspective on secant varieties via the notion of \textit{refined de Jonqui\` eres divisors}.

In the classical study of algebraic curves \textit{de Jonqui\`eres divisors} on a curve $C$ are divisors with prescribed multiplicities belonging to a linear series $l$ of degree $d$ and dimension $r$ (see for example \autocite[Chapter VIII, Section 5]{ACGH}).  More precisely, we say that an effective divisor $D$ belonging to $l$ is a de Jonqui\`eres divisor of length $N$ if there exist positive integers $a_1,\ldots,a_k$ and $d_1,\ldots,d_k$ with $\sum_{i=1}^k a_i d_i=d$ and $\sum_{i=1}^k d_i =N$ such that $D=\sum_{i=1}^k a_i D_i$, for some effective divisors $D_i$ of degree $d_i$.  We shall refer to those de Jonqui\`eres divisors where we fix not only the degrees $d_i$ of each of the $D_i$, but also the dimension of their spaces of global sections, as refined de Jonqui\`eres divisors.  Thus, if we additionally fix non-negative integers $r_1,\ldots,r_k\leq r$, a \textit{refined de Jonqui\` eres divisor of length} $N$ is a de Jonqui\`eres divisor $a_1 D_1 + \cdots + a_k D_k\in C_d$ of $l$ such that $h^0(C,\oo_C(D_i))\geq r_i +1$, for $i=1,\ldots,k$.  In this paper we shall investigate the geometry of spaces of refined de Jonqui\`eres divisors on a general curve, with particular emphasis on their dimension theory.  We shall use it to recover some of the already known results concerning secant varieties and to also provide some new insights into their study.

The connection between refined de Jonqui\`eres divisors and secant varieties arises from two distinct observations.  

On the one hand, consider the example found in \autocite[Remark 5.1]{Un2} that describes a situation where the expected dimension of  $V_e^{e-f}(l)$ is zero, the residual linear series $K_C-l$ is a pencil of degree $d-e$, and $r-e+f=1$, i.e.~a divisor $D\in V^{e-f}_e(l)$ imposes at most $r-1$ conditions on $l$.   Nevertheless we find that there are no divisors $D\in C_e$ satisfying $D'+D=l$, where $D'$ is an effective divisor belonging to $K_C-l$.  The situation there gives rise to a refined de Jonqui\`eres divisor $2D'+D$ of the canonical bundle $K_C$, where $h^0(C,\oo_C(D'))=2$.  Using dimension theoretical considerations, one may then easily check that the canonical bundle does not admit refined de Jonqui\`eres divisors of that type.

Inspired by this, we use the framework of refined de Jonqui\` eres divisors to produce similar examples for linear series $l$ of any dimension.  More precisely, we have
\begin{prop}\label{prop}
Let $C$ be a general curve of genus $g$ equipped with a complete linear series $l_1$ of degree $d$ and dimension $r$ with $g<d+1$ and its residual $l_2$.  There is no effective divisor $D\in C_{2g-2-2d}$ such that $l_1 + D = l_2$. 
\end{prop}

\noindent  We remark here that the example found in \autocite[Remark 5.1]{Un2} may also be recovered as a special case of Proposition \ref{prop}.  Furthermore, note that the above result does not imply that the corresponding secant variety $V^{g-d-1}_{2g-2-2d}(l_2)$ is empty within the stated parameter bound.  It does however emphasise the point that the expected non-emptiness of secant varieties cannot be used as a test for the existence of effective divisors $D$ satisfying $l_1+D=l_2$, for some fixed linear series $l_1$ and $l_2$.

On the other hand, suppose the linear series $l$ admits a refined de Jonqui\`eres divisor
$D' + D$, where $D\in C_e$, $D'\in C_{d-e}$, and $h^0(C,\oo_C(D')\geq r+1-e+f$.  This in turn means that $D$ is an element of the secant variety $V_e^{e-f}(l)$.  Hence, the non-emptiness of the space of such refined de Jonqui\`eres divisors implies the non-emptiness of the corresponding secant variety.  In this way we are able to establish the following 
\begin{thm}\label{thm1}
Let $C$ be a general curve of genus $g$ and $l$ a complete linear series of degree $d$ and dimension $r$.  Assume furthermore that $e,f$ are non-negative integers with $0\leq f<e\leq d$.  If
\[e-f(r+1-e+f)\geq (g-d+r-1)(r-e+f),\]  
then the secant variety $V_e^{e-f}(l)$ is non-empty.
\end{thm}

\noindent This result recovers some already known existence statements about secant varieties, and it provides some new (to us) instances of the non-emptiness of secant varieties when their expected dimension satisfies the lower bound from Theorem \ref{thm1}.  For example, when $g-d+r=1$ we recover one of the few cases where the MacDonald formula is easily seen to be positive, while for $g-d+r=2$ we are in the situation of the Coppens-Martens result mentioned above.

Finally, as we explained above, if $D$ is an effective divisor of degree $e\leq d$ such that $D' + D$ is a refined de Jonqui\`eres divisor with $h^0(C,\oo_C(D'))\geq r-e+f$, then $D$ also belongs to $V_e^{e-f}(l)$.  Elements $D$ of the secant variety $V_e^{e-f}(l)$ that originate from refined de Jonqui\`eres divisors give rise to a distinguished subvariety $\widetilde{V}_e^{e-f}(l)$ of $V_e^{e-f}(l)$ and we are establish the following

\begin{thm}\label{thm2}
Fix a general curve $C$ of genus $g$ equipped with a complete general linear series $l$ of degree $d$ and dimension $r$.  Assume also that $e,f$ are non-negative integers with $0\leq f<e\leq d$ and that $e-f(r+1-e+f)<\min(r-e+f+2,(g-d+r-1)(r-e+f))$. If
\begin{enumerate}
	\item $\rho(g,r,d)=0$ and $e<2r$ or if
	\item $\rho(g,r,d)>0$ and $e<r$, 
\end{enumerate}
then the subvariety $\widetilde{V}_e^{e-f}(l)$ of $V_e^{e-f}(l)$ is empty.
\end{thm}

\noindent The theorem is a direct consequence of the result concerning the expected dimension of the space of refined de Jonqui\`eres divisors of the type $D'+D$, with $D\in C_e$ and $D'\in C_e$ with $h^0(C,\oo_C(D'))\geq r+1-e+f$.

The paper is organised as follows.  In Section \ref{sec:definitions} we lay out the precise definitions of the objects of interest for this article.  We begin the study of the space of refined de Jonqui\`eres divisors in Section \ref{sec:refineddej} and we state the main results concerning its dimension.  In Section \ref{sec:app} we explain how to use refined de Jonqui\`eres divisors to make statements about secant varieties.  Thus, section \ref{sec:proofprop} is dedicated to the proof of Proposition \ref{prop} and Section \ref{sec:thm1} to the proof of Theorem \ref{thm1}.  Finally, in Section \ref{sec:degind} we use induction and degeneration techniques to study the dimension theory of the space of refined de Jonqui\`eres in certain relevant cases which then directly yields the content of Theorem \ref{thm2}.

\section*{Acknowledgements}
I am grateful to Ethan Cotterill and Gavril Farkas for pointing out several inaccuracies in an earlier version of this manuscript.

\section{Definitions and conventions}\label{sec:definitions}
Let $C$ be a smooth projective curve of genus $g$.  For a divisor $D$ on $C$, let $|D|$ be the linear series of all effective divisors linearly equivalent to $D$ and $s(D)$ be the index of speciality of $D$, equal to $g-d+\dim |D|$. Denote by $\Grd$ the space parametrising linear series of type $\grd$ on $C$, i.e.
\[ \Grd:=\{ l=(L,V) \mid L\in \pic^d(C), V\in Gr(r+1,H^0(C,L)) \}. \]
We focus on the case of Brill-Noether general curves, which means that $\Grd$ is a smooth variety of expected dimension given by the Brill-Noether number $\rho(g,r,d)$.  Furthermore, call the following sequence
\[ 0\leq a_0(l,p)<a_1(l,p)<\cdots<a_r(l,p)\leq d, \] 
where the $a_i(l,p)$ are the orders with which non-zero sections of $l$ vanish at a point $p\in C$ \textit{the vanishing sequence of $l$ at a point} $p$.  Associated to it we also have the \textit{ramification sequence of} $l$ at $p$:
\[\alpha(l,p):= 0\leq \alpha_0(l,p)\leq\alpha_1(l,p)\leq\cdots\leq\alpha_r(l,p)\leq d-r, \]
where $\alpha_i(l,p)=a_i(l,p)-i$.
  For a general pointed curve $(C,p)$, the variety $G^r_d(C;\alpha(l,p))$ parametrising $\grd$-s $l$ with prescribed ramification sequence at least $\alpha(l,p)$ at $p$ is also reduced and of expected dimension given by the adjusted Brill-Noether number $\rho(g,r,d)-\sum_{i=0}^r \alpha_i(l,p)$.

Let $\Crd$ be the subvariety of the symmetric product $C_d$ parametrising effective divisors moving in a linear series of dimension at least $r$:
\[ \Crd = \{ D\in C_d \mid \dim |D| \geq r \}. \]
It is a well-known fact (cf.~\autocite[Chapter IV]{ACGH}) that for a general curve, the variety $\Crd$ is of expected dimension $\rho(g,r,d)+r$ at a point $D\in \Crd\setminus C_d^{r+1}$.
 
As mentioned in the Introduction, $V_e^{e-f}(l)$ is the variety of effective divisors $D\in C_e$ imposing at most $e-f$ conditions on $l$, i.e.
\[ V_e^{e-f}(l) = \{ D\in C_e \mid \dim (l-D)\geq r-e+f \}\subset C_e, \]
where $l-D$ denotes the linear series $\bigl(L\otimes\oo_C(-D),V\cap H^0(C,L\otimes\oo_C(-D))\bigr)$ and $f<e$ are non-negative integers smaller than $d$.

Lastly, let $C$ be a smooth curve of genus $g$ equipped with a linear series $l=(L,V)\in\Grd$ and $r_1,\ldots, r_k\leq r$ non-negative integers, for some positive integer $k\leq d$.  A \textit{refined de Jonqui\`eres divisor of length} $N$ is a divisor $a_1 D_1 + \cdots + a_k D_k \in C_d$ such that
\[ a_1 D_1 + \cdots + a_k D_k \in \p V, \]
where the $D_i \in C_{d_i}^{r_i}$ and $\sum_{i=1}^k d_i = N$.  

Now fix the following vectors of non-negative integers $\mu=(r_1,\ldots,r_k)$, $\nu=(a_1,\ldots,a_k)$, and $\lambda=(d_1,\ldots,d_k)$ satisfying $\sum_{i=1}^k a_i d_i=d$, $\sum_{i=1}^k d_i = d$ and $r_i\leq r$ for all $i=1,\ldots,k$.  Denote by $DJ_{\lambda,\mu,\nu}(C,l)$ the space of refined de Jonqui\`eres divisors of length $N$ belonging to $l$ and determined by the three vectors $\lambda,\mu,\nu$.
Note that if $r_i=0$ for all $i=1,\ldots,k$, we recover the usual de Jonqui\`eres divisors of length $N$.  Furthermore, for a general curve $C$, one must choose the parameters in such a way that the inequality $\rho(g,r_i,d_i)\geq 0$ is satisfied for all $i=1,\ldots,k$.

\section{Dimension theory of refined de Jonqui\`eres divisors}\label{sec:refineddej}
As usual, let $C$ be a general curve of genus $g$ with a linear series $l=(L,V)\in\Grd$.
We begin the study of the space of refined de Jonqui\`eres divisors $DJ_{\lambda,\mu,\nu}(C,l)$ by first expressing it as a degeneracy locus of a vector bundle map over $C_d$.  Let $\mathcal{E}=\oo_{C_d}\otimes V$ be the trivial vector bundle.  Furthermore, let $\mathcal{F}_d(L)=\tau_* (\sigma^*L \otimes \oo_{\mathcal{U}})$ be the $d$-th secant bundle.

Consider the restriction map $\sigma^* L\rightarrow\sigma^* L\otimes \oo_{\mathcal{U}}$ and its pushdown via $\tau$ to $C_d$.  After intersecting with $\p V$, we obtain the desired vector bundle morphism which we denote by $\Phi:\mathcal{E}\rightarrow\mathcal{F}_d(L)$.

Furthermore set
\[ \Delta_{\lambda,\mu,\nu} := \{ E\in C_d \mid E=\sum_{i=1}^k a_i D_i \text{ for some }D_1\in C_{d_1}^{r_1},\ldots, D_k \in C_{d_k}^{r_k} \}\subset C_d. \]
It is easy to see that $\dim \Delta_{\lambda,\mu,\nu} \geq \sum_{i=1}^k (\rho(g,r_i,d_i) + r_i)$.  If moreover each $D_i$ belongs to $C^{r_i}_{d_i}\setminus C_{d_i}^{r_i+1}$, then 
\[ \dim \Delta_{\lambda,\mu,\nu} = \sum_{i=1}^k (\rho(g,r_i,d_i) + r_i). \]

Putting everything together, we see that $DJ_{\lambda,\mu,\nu}(C,l)$ is the degeneracy locus of the vector bundle map $\Phi$ over $\Delta_{\lambda,\mu,\nu}$.  Thus, if $DJ_{\lambda,\mu,\nu}(C,l)$ is non-empty, then for every point $D\in DJ_{\lambda,\mu,\nu}(C,l)$ we have that
\begin{equation}\label{eq:dimboundfirst}
\dim_D DJ_{\lambda,\mu,\nu}(C,l) \geq  \dim \Delta_{\lambda,\mu,\nu} -d+r.
\end{equation} 

We focus on the relevant case for our problem, namely
\[ \Delta_{\lambda,\mu,\nu} = \{ E\in C_d \mid E=D+D' \text{ for some }D\in C_e \text{ and }D'\in C_{d-e}^{r-e+f} \}, \] 
where now $\lambda=(e,d-e)$, $\mu=(0,r-e+f)$, and $\nu=(1,1)$.  To ease notation, from now on we shall refer to the space $\Delta_{\lambda,\mu,\nu}$ with this particular choice of parameters as $\Delta_{e,f}^{r,d}$ and the corresponding space of refined de Jonqui\`eres divisors as $DJ_{e,f}(C,l)$.
From the discussion in the Introduction it follows that if $DJ_{e,f}(C,l)$ is non-empty, then so is $V_e^{e-f}(l)$. 

Furthermore, \autocite[Chapter IV, Lemma 1.7]{ACGH} shows that no component of $\Crd$ is entirely contained in $C^{r+1}_d$.  We are thus interested in the dimension theory of the space $DJ_{e,f}(C,l)$ at a point $E=D+D'$, where $D\in C_e$ and $D'$ belongs to a linear series $l'\in G^{r-e+f}_{d-e}(C)$.  In this case, the expected dimension of $DJ_{e,f}(C,l)$ at a point $D$ is given by
\begin{equation}\label{eq:expdim}
\dim_D DJ_{e,f}(C,l) \geq  e-f(r+1-e+f)-(s-1)(r-e+f),
\end{equation} 
where $s$ denotes the index of speciality of $l$.

A large part of the paper is dedicated to proving the following dimensionality result.

\begin{thm}\label{thm:refineddjdim}
Fix a general curve $C$ of genus $g$ and a general complete linear series $l\in\Grd$ with index of speciality $s$.  Let $e,f$ be non-negative integers with $0\leq f<e\leq d$ and $e-f(r+1-e+f)\leq r-e+f+2$.  If
\begin{enumerate}
	\item $\rho(g,r,d)=0$ and $e<2r$ or if
	\item $\rho(g,r,d)>0$ and $e<r$,
\end{enumerate}
and if moreover the space $DJ_{e,f}(C,l)$ is non-empty, then it is of expected dimension $e-f(r+1-e+f)-(s-1)(r-e+f)$.
\end{thm}

\noindent We immediately obtain the following

\begin{cor}\label{cor:refineddjnonex}
Fix a general curve $C$ of genus $g$ and a general complete linear series $l\in\Grd$ with index of speciality $s$.   Let $e,f$ be non-negative integers with $0\leq f<e\leq d$ and $e-f(r+1-e+f)\leq r-e+f+2$.  If  
\begin{enumerate}
	\item $\rho(g,r,d)=0$ and $e<2r$ or if
	\item $\rho(g,r,d)>0$ and $e<r$,
\end{enumerate}
and if $e-f(r+1-e+f)<\min((s-1)(r-e+f),r-e+f+2)$, then $DJ_{e,f}(C,l)$ is empty.
\end{cor}
The statement of Theorem \ref{thm1} is a direct consequence of the above Corollary and the discussion above.   

The strategy for the proof of Theorem \ref{thm:refineddjdim} is to first consider the case of the canonical linear series $K_C$ and that of non-special linear series.  This is done in Section \ref{sec:transvers} below, by means of a tangent space computation.  We then use this as a base case for an induction argument that we explain in Section \ref{sec:degind}.  Along the way we describe what happens to a family of refined de Jonqui\`eres divisors that degenerates to a nodal central fibre.  We focus in particular on the case of a nodal curve of compact type with one node.  We remark here that the bounds $e<2r$ and $e<r$ are artefacts of the induction step and the particular degeneration that we use.

Finally, to obtain Corollary \ref{cor:refineddjnonex} we need to establish the emptiness of the space $DJ_{e,f}(C,l)$ when $l$ is a non-special linear series (so with $s=0$) and when $l=K_C$.  The latter is discussed at the beginning of the next section, while for the former we have:
\begin{lemma}\label{lemma:nonspecial}
Fix a general curve $C$ of genus $g$, a linear series $l\in\Grd$ with index of speciality $s=0$, and non-negative integers $e,f$ with $0\leq f<e\leq d$.  If $e-f(r+1-e+f)+(r-e+f)<0$, then $DJ_{e,f}(C,l)$ is empty. 
\end{lemma}
\begin{proof}
Consider the Abel-Jacobi map $u: \Delta_{e,f}^{r,d}\rightarrow \pic^d(C)$, $E\mapsto \oo_C(E)$.  Since $s=0$, we have $\pic^d(C)=\Wrd$.  Moreover, $\im u$ is a closed subset of $\pic^d(C)$ and $\dim \im u \leq \dim\Delta_{e,f}^{r,d}$.  Thus if $\dim \Delta^{r,d}_{e,f} <g=d-r$, then a general $L\in pic^d(C)$ (i.e. such that $s=0$) does not lie in the image of $u$ and hence its corresponding space of refined de Jonqui\`eres divisors $DJ_{e,f}(C,L)$ is empty.  One then easily verifies that  $\dim \Delta^{r,d}_{e,f} <d-r$ is equivalent to the inequality in the statement of the lemma. 
\end{proof}

\subsection{Tangent space of refined de Jonqui\`eres divisors}\label{sec:transvers}
We begin with an elementary argument which shows that Theorem \ref{thm:refineddjdim} holds for the canonical linear series $K_C$.  We then prove that in fact $DJ_{\lambda,\mu,\nu}(C,K_C)$ is always of expected dimension by looking at its tangent space.

Consider the special case of refined de Jonqui\`eres divisors corresponding to $l=K_C$, $\lambda=(d-d_2,d_2)$, $\mu=(0,r_2)$, and $\nu=(1,1)$, that is
\begin{equation}\label{def:delta}
\Delta_{\lambda,\mu,\nu} = \{ E\in C_d \mid E=D_1 + D_2, \text{ for some }D_1\in C_{d-d_2}, D_2\in C^{r_2}_{d_2} \}.
\end{equation} 
From Riemann-Roch we immediately obtain that $\dim |D_1| = g-d_2+r_2-1$.  Thus, up to linear equivalence, we get that for each $D_2\in C^{r_2}_{d_2}$, there is a $(g-d_2+r_2-1)$-dimensional family of divisors $D_1$ such that $D_1+D_2=K_C$.  Thus in this case we have
\begin{align*}
\dim DJ_{\lambda,\mu,\nu}(C,K_C) &= \rho(g,r_2,d_2) + g-d_2 +r_2 -1\\
&= (\rho(g,r_2,d_2) +r_2) + (d-d_2) - d +r,
\end{align*}
which is exactly the expected dimension of the space of refined de Jonqui\`eres divisors from (\ref{eq:dimboundfirst}), with $d=2g-2$ and $r=g-1$.  Thus Theorem \ref{thm:refineddjdim} holds for $K_C$ and we may also at the same time conclude that if the expected dimension is negative then the space $DJ_{\lambda,\mu,\nu}(C,K_C)$ is empty. 

\begin{rem}\label{rem:ramif}
Assume now that $D_2$ belongs to a linear series $l_2$ of type $\grdop{r_2}{d_2}$ with prescribed ramification sequence at least $\alpha(l_2,p)$ at a general point $p\in C$.  By appropriately modifying the expression in (\ref{def:delta}) to include this further restriction, we see that the dimension of the space $DJ_{\lambda,\mu,\nu,\alpha}(C,K_C)$ of refined de Jonqui\`eres divisors with prescribed ramification at a general point $p\in C$ for $K_C$ is
\begin{equation}\label{eq:dimprescribedramif}
\dim DJ_{\lambda,\mu,\nu,\alpha}(C,K_C)=(\rho(g,r_2,d_2) -\sum_{i=0}^{r_2}\alpha_i(l_2,p) +r_2) + (d-d_2) - d +r.
\end{equation}
\end{rem}

\begin{rem}\label{rem:2}
In proving Theorem \ref{thm:refineddjdim} we shall in fact need to consider refined de Jonqui\`eres divisors of type $DJ_{e,f}(C,l)$ where the corresponding linear series $l'\in G^{r-e+f}_{d-e}(C)$ has ramification at least $\alpha$ at a general point $q\in C$.  From now on we denote the space of such refined de Jonqui\`eres divisors by $DJ_{e,f,\alpha}(C,l)$ and note that its expected dimension is given by
\[  \dim DJ_{e,f,\alpha}(C,l) \geq e-f(r+1-e+f)-(s-1)(r-e+f)-\sum_{i=0}^{r-e+f}\alpha_i.\]  
\end{rem}

We expect that $DJ_{\lambda,\mu,\nu}(C,K_C)$ is of expected dimension for a more general choice of parameters $\lambda$, $\mu$, and $\nu$ than the one appearing in Theorem \ref{thm:refineddjdim}.  This is indeed what we find by means of a tangent space computation, which we now discuss.

Let $l\in\Grd$ be a complete linear series and let $D\in C_d$ be a divisor such that $|D|=l$. We can rewrite $DJ_{\lambda,\mu,\nu}(C,l)$ as the intersection
\[ DJ_{\lambda,\mu,\nu}(C,l) = |D| \cap \Delta_{\lambda,\mu,\nu}. \]
Thus, the space $DJ_{\lambda,\mu,\nu}(C,l)$ is smooth and of expected dimension at a point $D$ if and only if the above intersection is transverse, i.e.
\[ T_D C_d = T_D |D| + T_D \Delta_{\lambda,\mu,\nu}. \]

Recall that $T_D C_d = H^0(C,\oo_D(D))$ (see for example \S 1, Chapter IV of \cite{ACGH}).  Its dual is $T_D^{\vee}C_d = H^0(C,K_C/K_C(-D))$ and the pairing between the tangent and cotangent space is given by the residue.  From now on we shall use the superscript $^{\perp}$ to denote orthogonality with respect to this pairing.  Moreover, we have that
\[ T_D |D| = \ker \delta = \im(\alpha)^{\perp}, \]
where $\delta:\im(\alpha\mu_0)^{\perp}\rightarrow \im(\mu_0)^{\perp}$ is the differential of the restriction of the Abel-Jacobi map $u:\Crd\rightarrow\Wrd$, the map $\alpha:H^0(C,K_C)\rightarrow H^0(C,K_C\otimes\oo_D)$ is the restriction, and 
\[ \mu_0:H^0(C,K_C-D)\otimes H^0(C,\oo_C(D))\rightarrow H^0(C,K_C) \]
is the cup-product mapping.
We also used the fact that $T_D C^r_d = \im(\alpha\mu_0)^{\perp}$ and denoted by $\Wrd$ the subvariety of $\pic^d(C)$ parametrising complete linear series of degree $d$ and dimension at least $r$, that is $\Wrd=\{ L\in\pic^d(C) \mid h^0(C,L)\geq r+1 \}$.
To establish the expression for the tangent space of $\Delta_{\lambda,\mu,\nu}$, let $\mathcal{D}_i$ denote the diagonal in the $a_i$-th product $C^{r_i}_{d_i}\times\cdots C^{r_i}_{d_i}$.  Thus $\Delta_{\lambda,\mu,\nu}=\mathcal{D}_1\times\cdots\times\mathcal{D}_k / S_d$ and so
\begin{align*}
 T_D \Delta_{\lambda,\mu,\nu} &= T_{D_1} C_{d_1}^{r_1} \oplus \cdots \oplus T_{D_k} C_{d_k}^{r_k} \\
 &=\im(\alpha_{1}\mu_{01})^{\perp_1}\oplus\cdots\oplus \im(\alpha_k\mu_{0k})^{\perp_k},
\end{align*}
where $\mu_{0i}$ and $\perp_i$ denote the cup-product mapping and orthogonality with respect to the residue pairing corresponding to each divisor $D_i$, for $i=1,\ldots,k$. 
Thus the transversality condition becomes
\begin{equation}\label{eq:transcond1}
\im(\alpha_{1}\mu_{01})^{\perp_1}\oplus\cdots\oplus \im(\alpha_k\mu_{0k})^{\perp_k}+\im(\alpha)^{\perp} = T_D C_d.
\end{equation}
Note that if $|D|=K_C$, then condition (\ref{eq:transcond1}) is immediately satisfied. 
We have therefore proved the following:
\begin{lemma}\label{lemma:trans}
Fix a general curve $C$ with a linear series $l$.  If $l=K_C$ or $l$ is non-special, then if $DJ_{\lambda,\mu,\nu}(C,l)$ is non-empty, then it is smooth and of expected dimension. 
\end{lemma}

\section{Applications to secant varieties}\label{sec:app}
This section is dedicated to explaining how one may use refined de Jonqui\`eres divisors in order to extract information about secant varieties on algebraic curves.

Note that the multiplication map
\begin{align*}
\epsilon:C^{r-e+f}_{d-e}\times C_e &\rightarrow \Delta_{e,f}^{r,d}\\
(D,D')&\mapsto D+D',
\end{align*} is a finite morphism.  Hence by projecting the preimage $\epsilon^{-1}(\Delta_{e,f}^{r,d}\cap |l|)$ to $C_e$ we obtain a subvariety $\widetilde{V}_e^{e-f}(l)$ of $V_e^{e-f}(l)$ of dimension equal to at most the dimension of $DJ_{e,f}(C,l)$.

At this point it is also important to understand to what extent the varieties $V_e^{e-f}(l)$ and $V_e^{e-f-1}(l)$ are different.  More precisely, we have the following
\begin{lemma}
Let $l\in\Grd$ be a complete linear series on a general curve $C$.  Then no component of $V_e^{e-f}(l)$ is entirely contained in $V_e^{e-f-1}(l)$.
\end{lemma}
\begin{proof}
Out proof is similar to that of \autocite[Chapter IV, Lemma 1.7]{ACGH}.  Let $D$ be a general point of an irreducible component of $V_e^{e-f}(l)$ and assume towards a contradiction that $\dim (l-D)>r-e+f$.  Let $q$ be a general point of $C$.  Then there exists a divisor $E\in C_{d-e-1}$ such that $E+q \in l-D$.  Moreover, since $q$ is general (and thus not a base point of $l-D$), we obtain $\dim |E| = \dim (l-D)-1\geq r-e+f$.  We therefore have that 
\[ s(E)=g-(d-e-1)+(\dim (l-D)-1)=g-d+\dim(l-D)+e\geq g-d+r+f, \]
where $s(E)$ is the index of speciality of $|E|$.  Hence $K_C-E$ is effective.  Furthermore, if $p$ is another general point of $C$, then $p$ is not a base point of $|K_C-E|$ so that $s(E+p)=s(E)-1$.  Riemann-Roch then yields $\dim|E+p| = \dim |E| =\dim(l-D)-1\geq r-e+f$.  Next, observe that $E+p\sim l-D+p-q$.  Let $D'$ be an effective divisor linearly equivalent to $D-p+q$.  By the generality of $p$ and $q$, $D'$ belongs to the same irreducible component of $V_e^{e-f}(l)$ as $D$ and moreover $\dim (l-D')\geq r-e+f$, which contradicts the generality of $D$.
\end{proof}

Thus it is sensible to consider, as we have done already in Section \ref{sec:refineddej} the dimension theory of $DJ_{e,f}(C,l)$ at a point $E=D+D'$, where $D\in C_e$ and $D'\in C^{r-e+f}_{d-e}\setminus C^{r+1-e+f}_{d-e}$.

In what follows we shall prove Proposition \ref{prop} and Theorem \ref{sec:thm1} using the above considerations and the transversality result in Lemma \ref{lemma:trans}.  The proof of Theorem \ref{thm:refineddjdim} (and hence that of Theorem \ref{thm2}) requires a degeneration argument which is described in Section \ref{sec:degind}.

\subsection{Proof of Proposition \ref{prop}}\label{sec:proofprop}
As mentioned in the Introduction, the starting point for the present discussion is the example found in \autocite[Remark 5.1]{Un2}.  Using the notation of Proposition \ref{prop}, the situation there can be described as follows: $l_1\in G^1_d(C)$ is a minimal pencil, i.e.~$\rho(g,1,d)=1$.  This means that the general curve $C$ must have genus $g=2d-3$, $l_2\in G^{d-3}_{3d-8}(C)$, and $K_C$ is of type $\grdop{2d-4}{4d-8}$.  Let $D'$ be an effective divisor such that $|D'|=l_1$.  The idea is that, if there exists an effective divisor $D\in C_{2d-8}$ with $l_1+D=l_2$, then $D=K_C-2D'\geq 0$.  Consider now the Petri map
\[ \mu_0:H^0(C,D')\otimes H^0(C,K_C-D')\rightarrow H^0(C,K_C). \]
The base-point-free pencil trick yields that its kernel is given by
$H^0(C,K_C-2D')$. 
Since for a general curve $\mu_0$ is injective, we have that $h^0(C,K_C-2D')=0$, from which we conclude that $h^0(C,D)=0$.  But this contradicts the effectiveness of $D$, which means that such a divisor $D$ cannot exist.

To prove Proposition \ref{prop}, assume from now on that $l_1\in\Grd$ with $r\geq 1$.  As in the case of pencils, if there exists an effective divisor $D$ satisfying $l_1+D=l_2$, where $l_2=K_C-l_1$, then $K_C=2D'+D$, where $D\in C_{2g-2d-2}$ and $D'\in C_d$ and with $|D'|=l_1$.  Hence $2D'+D\in DJ_{\lambda,\mu,\nu}(C,K_C)$, where $\lambda=(2g-2d-2,d)$, $\mu=(r,0)$, and $\nu=(2,1)$.
	
From Lemma \ref{lemma:trans} we know that, if non-empty, the space of refined de Jonqui\`eres divisors $DJ_{\lambda,\mu,\nu}(C,K_C)$ is of expected dimension, i.e.
\begin{align*}
\dim DJ_{\lambda,\mu,\nu}(C,K_C) &= (\rho(g,r,d)+r) + (2g-2d-2) - (2g-2) +(g-1)\\
&=(g-d-1)(1-r) - r(r+1).
\end{align*} 
Lemma \ref{lemma:trans} also allows us to conclude that, if $(g-d-1)(1-r) - r(r+1)\geq 0$, then $DJ_{\lambda,\mu,\nu}(C,K_C)$ is non-empty.
It is however not difficult to check that the inequality is never satisfied.  For $r=1$, this is immediately clear.  If $r>1$, then the inequality is equivalent to 
\[ \frac{r(r+1)}{r-1}\leq d+1-g \]
which translates to
\[ \frac{2r}{r-1}\leq 1-s, \]
where, as usual, $s=g-d+r$ is the index of speciality of $l_1$.  However, the above inequality can never be satisfied for $s\geq 0$ and $r>1$.

Hence $\dim DJ_{\lambda,\mu,\nu}(C,K_C)<0$ and we therefore expect that it is empty.  We establish that it is indeed the case for $g< d+1$.  By construction $l_2=K_C-l_1=l_1+D$.  We may rewrite this as 
\[ \grdop{s-1}{2g-2-d} = \grd + D. \]
Note that if $g=d+1$, then $s-1=r$ and $2g-2-d=d$, which yields $D=0$.  The divisors $D'$ are therefore theta characteristics and we recover the fact that there are finitely many odd ones on a general curve. 

Now consider the secant variety relevant to our situation: we have that $l_2$ is a linear series of type $\grdop{g-d+r-1}{2g-2-d}$ and $e=\deg D=2g-2d-2$, which means that
\[f=r+e-(g-d+r-1)=g-d-1.\]
Thus the relevant secant variety is $V_{2g-2d-2}^{g-d-1}(l_2)$ which (if it is non-empty) is of expected dimension
\[\text{exp}\dim V_{2g-2d-2}^{g-d-1}(l_2) = (g-d-1)(1-r)\geq 0.\]
Note that for this to make sense, we must have $g\leq d+1$.
Thus we have shown that when $g<d+1$ and therefore $(g-d-1)(1-r)>0$
holds, the above secant variety is (expected to be) non-empty, but we find examples of linear series $l_1\in\Grd$ and $l_2=K_C-l_1$ such that there are no effective divisors $D$ satisfying $l_1 + D = l_2$.

\subsection{Proof of Theorem \ref{thm1}}\label{sec:thm1}
From the discussion at the beginning of Section \ref{sec:app}, we know that if we can establish the non-emptiness of the subvariety of $\widetilde{V}_e^{e-f}(l)$, then the non-emptiness of $V_e^{e-f}(l)$ follows automatically.  We do this by reducing the problem to the case when $l=K_C$.

More precisely, assume $l\in\Grd$ is a complete linear series with index of speciality $s=g-d+r$.  Then $l$ has a residual linear series $K_C-l\in G^{s-1}_{2g-2-d}$.  Let $E$ be an effective divisor in $|K_C - l|$.  Hence if $D+D'\in DJ_{e,f}(C,l)$ such that $D'\in C^{r-e+f}_{d-e}\setminus C^{r-e+f+1}_{d-e}$, then 
\[ D'+D+E=K_C. \]
Thus $D'+D+E$ is now a refined de Jonqui\`eres divisor corresponding to $K_C$, that is $D'+D+E\in DJ_{e',f'}(C,K_C)$, where
\begin{align*}
&e'=e+2g-2-d\\
&f'=f+s-1.
\end{align*}
From Lemma \ref{lemma:trans} we have that
\begin{align*}
\dim DJ_{e',f'}(C,K_C)&=e+2g-2-d-(f+s-1)(r+1-e+f)\\
&=e-f(r+1-e+f)-(s-1)(r-e+f)-(s-1)+(2g-2-d)
\end{align*}  
Note that the terms $s-1$ and $2g-2-d$ account for the presence of $E$ in the refined de Jonqui\`eres divisors under consideration.  We conclude that if the space $DJ_{e',f'}(C,K_C)$ is non-empty, then so is $DJ_{e,f}(C,l)$.  We then use Lemma \ref{lemma:trans} for $K_C$ once more to see that, indeed, if $e-f(r+1-e+f)-(s-1)(r-e+f)\geq 0$, then the space $DJ_{e,f}(C,l)$ is not empty.

\section{Degeneration and induction argument}\label{sec:degind}
This section is dedicated to the study of degenerations of de Jonqui\`eres divisors in the simple case of families of curves where the special fibre is a reducible curve of compact type.  To keep notation to a minimum, we moreover restrict to the case of refined de Jonqui\`eres divisors corresponding to $\Delta_{e,f}^{r,d}$.  We discuss this in Section \ref{sec:degeneration} after a brief review in Section \ref{sec:llsreview} of the limit linear series tools that are needed.  We then use this framework to prove Theorem \ref{thm:refineddjdim} in Section \ref{sec:proofrefineddjdim}.

\subsection{Limit linear series review}\label{sec:llsreview}
Let $\pi:\mathscr{X}\rightarrow B$ be a family of curves of genus $g$ that has a section $\sigma:B\rightarrow \mathscr{X}$ and where $B=\text{Spec}(R)$ for some discrete valuation ring $R$ with uniformising parameter $t$.  Moreover, let $\mathscr{X}$ be a nonsingular surface, projective over $B$ and let $0\in B$ denote the point corresponding to the maximal ideal of $R$ and $\eta$ the geometric generic point of $B$.  Furthermore assume that the special fibre $\mathscr{X}_0$ is a reduced curve of compact type, while the generic fibre is a smooth, irreducible curve. 

After (possibly) performing a base change, a series $(\mathscr{L}_{\eta},\mathscr{V}_{\eta})$  of type $\grd$ on $\mathscr{X}_{\eta}$ determines a $\grd$ on each irreducible component $Y$ of $\mathscr{X}_0$ as follows: by the smoothness of $\mathscr{X}$, $\mathscr{L}_{\eta}$ extends to a (unique up to tensoring with a Cartier divisor supported on $\mathscr{X}_0$) line bundle on $\mathscr{X}$.  Denote by $\mathscr{L}_Y$ the extension of $\mathscr{L}_{\eta}$ with the property that $\deg(\mathscr{L}_Y|_Y)=d$ and $\deg(\mathscr{L}_Y|_Z)=0$ for any irreducible component $Z\neq Y$ of $\mathscr{X}_0$.  Set $\mathscr{V}_Y:=(\mathscr{V}_{\eta}\cap\pi_* \mathscr{L}_Y)\otimes k(0)$.  One sees that
\[ \mathscr{V}_Y \simeq \pi_* \mathscr{L}_Y \otimes k(0) \subset H^0(\mathscr{X}_0,\mathscr{L}_Y |_{\mathscr{X}_0}) \]
is a vector space of dimension $r+1$ which we shall identify with its image in $H^0(Y,\mathscr{L}_Y|_Y)$.  Thus the pair $(\mathscr{L}_Y|_Y,\mathscr{V}_Y)$ is a $\grd$ on $Y$ that we call the $Y$-aspect of the series $(\mathscr{L}_{\eta},\mathscr{V}_{\eta})$.  Moreover, we call the collection of aspects 
\[ l=\{(\mathscr{L}_Y|_Y,\mathscr{V}_Y)  \mid Y \text{ component of }\mathscr{X}_0  \} \]
the limit of $(\mathscr{L}_{\eta},\mathscr{V}_{\eta})$.

Now let $X$ be a curve of compact type and set
\[ l:=\{ l_Y \text{ a }\grd \text{ on }Y \mid Y \text{ irreducible component of }X \}. \]
We say that $l$ is a crude limit linear series on $X$ if the following compatibility condition on the vanishing sequences of the sections belonging to the aspects holds: for any irreducible component $Z\neq Y$ of $X$ with $Y\cap Z=p$ and for all $i=0,\ldots,r$,
\begin{equation}\label{eq:vanishinglls}
a_i(l_Y,p) + a_{r-i}(l_Z,p)\geq d.
\end{equation}  
If equality holds in (\ref{eq:vanishinglls}), then $l$ is called a refined limit linear series.
In \autocite{EH86} Eisenbud and Harris show that the limit $(\mathscr{L}_{\eta},\mathscr{V}_{\eta})$ above is a crude limit linear series.  

Unfortunately, it is not always true that any limit linear series arises as the limit of a series $(\mathscr{L}_{\eta},\mathscr{V}_{\eta})$.  We can however establish when this holds in certain cases.  We record the relevant results in this direction which are needed for the proof of Theorem \ref{thm:refineddjdim}.  In particular, let $X=C_1\cup_p C_2$ be a curve of compact type consisting of two irreducible components $C_1$, $C_2$ of genus $g_1$ and $g_2$.  Assume the two pointed curves $(C_1,p)$ and $(C_2,p)$ are general.  Let $l=\{ (C_1,l_{C_1}),(C_2,l_{C_2}) \}$ be a limit linear series of type $\grd$ on $X$, where each aspect $l_{C_i}$ has prescribed ramification $\alpha(l_{C_i},p)$ at $p$, for $i=1,2$.  From \autocite[Proposition 1.2]{EH87} it follows that the components $(C_i,p)$ may carry an $l_{C_i}$ with prescribed ramification $\alpha_i(l_{C_i},p)$ at $p$ if and only if, for $i=1,2$:
\begin{equation}\label{eq:existencells}
\sum_{j=0}^r(\alpha_j(l_{C_i},p)+g_1-d+r)_+ \leq g_i,
\end{equation} 
where $(x)_+=\max\{x,0\}$.
In addition, the limit linear series $l$ must be refined in order to satisfy the hypotheses of the smoothability theorem of Eisenbud and Harris in \autocite[Theorem 3.4]{EH86}.  This together with the dimension theorem \autocite[Theorem 1.2]{EH87} ensure that a refined limit linear series $l$ on $X$ as above is smoothable.

In what follows we shall also use the limit linear series framework developed by Osserman in a series of articles starting with \autocite{Os}, which gives rise to equivalent structures in the case of refined limit linear series.  Below we summarise the aspects of this approach that are most important to us.

Let $B$ be a scheme, $T$ a $B$-scheme, and $\mathscr{X}\rightarrow B$ a smooth proper family of smooth curves of fixed genus $g$ with a section.  Consider pairs of the form $(\mathscr{L},\mathscr{V})$, where $\mathscr{L}$ is a line bundle of relative degree $d$ on $\mathscr{X}\times_B T$ and $\mathscr{V}\subseteq\pi_{2*}\mathscr{L}$ is a subbundle of rank $r+1$ and where $\pi_2:\mathscr{X}\times_B T\rightarrow T$ is the projection.  We say that $(\mathscr{L},\mathscr{V})$ and $(\mathscr{L}',\mathscr{V}')$ are equivalent if there exists a line bundle $\mathscr{M}$ on $T$ and an isomorphism $\varphi:\mathscr{L}\rightarrow\mathscr{L}'\otimes\pi_2^* \mathscr{M}$ with the property that $\pi_{2*}\varphi$ maps $\mathscr{V}$ into $\mathscr{V}'$.   
 We define the functor $\mathscr{G}^r_d(\mathscr{X}/B)$ of linear series of type $\grd$ by associating to each $T$ the set of equivalence classes of pairs $(\mathscr{L},\mathscr{V})$ with respect to the above equivalence relation.  This functor is represented by a scheme $G^r_d(\mathscr{X}/B)$ which is proper over $B$.
 
Now suppose $\mathscr{X}\rightarrow B$ is a flat proper family of nodal curves of compact type of fixed genus $g$ where no nodes are smoothed.  All fibres therefore have the same dual graph $\Gamma$.  Denote the irreducible component of the fibre $\mathscr{X}_t$ corresponding to the vertex $v\in V(\Gamma)$ by $\mathscr{Y}^v_t$.  Hence, for each vertex $v$ of $\Gamma$ we have a family of smooth curves $\mathscr{Y}\rightarrow B$.  To define the functor of linear series of type $\grd$ for the family $\mathscr{X}$, consider the product $\prod_v \mathscr{G}^r_d(\mathscr{Y}^v /B)$ fibred over $B$, whose $T$-valued points consist of tuples of pairs $(\mathscr{L}^v,\mathscr{V}^v)$ with $\mathscr{L}^v$ a line bundle of relative degree $d$ on $\mathscr{Y}^v\times_B T$ and $\mathscr{V}^v\subseteq\pi_{2*}\mathscr{L}^v$ a subbundle of rank $r+1$.  For a line bundle $\mathscr{L}$ on $\mathscr{X}\times_B T$, denote by $\mathscr{L}^{\vec{d}}$ the twist of $\mathscr{L}$ that has multidegree given by the vector $\vec{d}=(d_v)_{v\in V(\Gamma)}$ of positive integers.  Let $\vec{d}^v=(0,\ldots,0,d_v,0,\ldots,0)$ be the vector of integers with entry $d_v$ at $v$ and 0 elsewhere.  Moreover, given two distinct vectors $\vec{d}$ and $\vec{d}'$, let $f_{\vec{d},\vec{d'}}:\mathscr{L}^{\vec{d}}\rightarrow\mathscr{L}^{\vec{d}'}$ be the unique map obtained by performing the minimal amount of bundle twists.  We now define the functor $\mathscr{G}^r_d(\mathscr{X}/B)$: consider a $T$-valued point of $\prod_v \mathscr{G}^r_d(\mathscr{Y}/B)$, namely a tuple of pairs $(\mathscr{L}^v,\mathscr{V}^v)_{v\in V(\Gamma)}$.  Let $\mathscr{L}$ be the line bundle on $\mathscr{X}\times_B T$ of multidegree $\vec{d}^{v_0}$ induced by $\mathscr{L}^{v_0}$ and appropriate twists of $\mathscr{L}^v$, for $v\neq v_0$, where $v_0$ is a fixed vertex of $\Gamma$.   
Then $(\mathscr{L}^v,\mathscr{V}^v)_{v\in V(\Gamma)}$ belongs to $\mathscr{G}^r_d(\mathscr{X}/B)(T)$ if, for all multidegrees $\vec{d}$, the map
\[ \pi_{2*}\mathscr{L}^{\vec{d}}\rightarrow\bigoplus_v (\pi_{2*}\mathscr{L}^{v})/\mathscr{V}^{v} \]
induced by the $f_{\vec{d},\vec{d}^v}$ and restriction to $\mathscr{Y}^v$ has its $(r+1)$st vanishing locus equal to all of $T$.   This functor also is also represented by a scheme $G^r_d(\mathscr{X}/B)$ that is proper over $B$.

Suppose finally that $\mathscr{X}\rightarrow B$ is a family of curves as above, but with some nodes that are smoothed.  It follows that the dual graph of the fibres is not constant and the components $\mathscr{Y}^v$ may not be defined over all of $B$.  Assume therefore that there exists a maximally degenerate fibre over some $b_0\in B$ with dual graph $\Gamma_0$ and fix a vertex $v_0\in V(\Gamma_0)$.  If such a family satisfies additionally certain technical conditions detailed in \autocite[Definition 3.1]{Os}, then it is called a \textit{smoothing family}.  Consider tuples $(\mathscr{L},(\mathscr{V}^v)_{v\in V(\Gamma_0)})$ where $\mathscr{L}$ is a line bundle of multidegree $\vec{d}^{v_0}$ on $\mathscr{X}\times_B T$ and the $\mathscr{V}^v$ are subbundles of the twists $\pi_{2*}\mathscr{L}^{\vec{d}^v}$ of rank $r+1$, for each $v\in V(\Gamma_0)$.  The functor $\mathscr{G}^r_d(\mathscr{X}/B)$ of limit linear series of type $\grd$ is defined in this case as follows: a $T$-valued point $(\mathscr{L},(\mathscr{V}^v)_{v\in V(\Gamma_0)})$ is in $\mathscr{G}^r_d(\mathscr{X}/B)(T)$ if for an open cover $\{U_k\}_{k\in I}$ of $B$ arising from the enriched structure of the curves (see discussion in \autocite{BH}), for all $k\in I$ and for all multidegrees $\vec{d}$, the map
\[ \pi_{2*}\mathscr{L}^{\vec{d}}|_{f\circ\pi_2^{-1}(U_k)}\rightarrow \bigoplus_v \Biggl( \pi_{2*}\mathscr{L}^{\vec{d}^v}|_{f\circ\pi_2^{-1}(U_k)} \Biggr)/\mathscr{V}^v|_{f^{-1}(U_k)}, \]
induced by the local versions of the twist maps and where $f$ is the structural morphism from $T$ to $B$, has its $(r+1)$st degeneracy locus equal to the whole of $U_k$.  Finally, one can show that the definition above is independent of choice of open cover $\{U_k\}_{k\in I}$, twist maps, and vertex $v_0$.

We remark here that all constructions from above are compatible with base change and the fibre over a point $t\in B$ is a space of refined Eisenbud-Harris limit linear series when $\mathscr{X}_t$ is nodal and a space of usual linear series when $\mathscr{X}_t$ is smooth.

\subsection{Refined de Jonqui\`eres divisors for reducible curves with one node}\label{sec:degeneration}
Let $l$ be a refined limit linear series on a curve of compact type with only two components $X=C_1\cup_{p}C_2$.  Assume furthermore that this linear series is smoothable, that is there exists a family of curves $\pi:\mathscr{X}\rightarrow B$ with central fibre $\mathscr{X}_0=X$ and a linear series $(\mathscr{L}_{\eta},\mathscr{V}_{\eta})$ of type $\grd$ whose limit is $l$.  Let $Y\subset X$ be an irreducible component of $X$ and $\mathscr{E}_{\eta}=(\sigma)\in |\mathscr{L}_{\eta}|$ a divisor on $\mathscr{X}_{\eta}$, where $\sigma$ is a section of $\mathscr{X}_{\eta}$.  Multiply $\sigma$ by the unique power of $t\in B_{\eta}$ so that it extends to a holomorphic section $\sigma_Y$ of the extension $\mathscr{L}_Y$ to the whole of $\mathscr{X}$ and so that it does not vanish identically on $\mathscr{X}_0$.  We find that the limit of the divisor $\mathscr{E}_{\eta}$ on $X$ is the divisor $(\sigma_Y)|_Y$.  

Now assume that $\mathscr{E}_{\eta}$ is a refined de Jonqui\`eres divisor in $\Delta_{e,f}^{r,d}$ for $\mathscr{L}_{\eta}$ on $\mathscr{X}_{\eta}$.  In other words, $\mathscr{E}_{\eta}=\mathscr{D}_{\eta}+\mathscr{D}'_{\eta}$, where $\mathscr{D}_{\eta}$ is an effective divisor of degree $e$ on $\mathscr{X}_{\eta}$ that cuts out a sublinear series 
\[(\mathscr{L}'_{\eta},\mathscr{V}'_{\eta}):=(\mathscr{L}_{\eta}\otimes\oo_{\mathscr{X}_{\eta}}(-\mathscr{D}_{\eta}),\mathscr{V}_{\eta}\cap H^0(\mathscr{X}_{\eta},\mathscr{L}_{\eta}\otimes\oo_{\mathscr{X}_{\eta}}(-\mathscr{D}_{\eta})))\] 
of type $\grdop{r-e+f}{d-e}$ on $\mathscr{X}_{\eta}$ and that contains the effective divisor $\mathscr{D}'_{\eta}$. More precisely we use the identification $|\mathscr{D}'_{\eta}|=\p H^0(\mathscr{X}_{\eta},\mathscr{L}'_{\eta})$ to say that $\mathscr{D}'_{\eta}$ is an element of $\p\mathscr{V}'$.  Note that after possibly performing a base change and resolving the resulting singularities, the pair $(\mathscr{L}'_{\eta},\mathscr{V}'_{\eta})$ also induces a refined limit linear series $l'$ of type $\grdop{r-e+f}{d-e}$ on $X$.

We want to understand what the refined de Jonqui\`eres divisors on the central fibre $X$ look like.    To this end we make the following

\begin{defi}\label{def:refineddj}
Let $X=C_1 \cup_p C_2$ be a nodal curve of compact type with a refined limit linear series $l$ of type $\grd$.  The divisor $E=D+D'$ on $X$ with smooth support and such that $D\in X_e$ and $D'\in X_{d-e}$ is a \textit{refined de Jonqui\`eres divisor of type} $\Delta_{e,f}^{r,d}$ for $l$ if for each irreducible component $C_i$, where $i=1,2$ both conditions below hold:
\begin{enumerate}[label=\roman*)]
	\item The aspect $l_{C_i}$ has a section vanishing on $E_i=D_i + D'_i$, where $E_i$, $D_i$, $D'_i$ denote the restrictions of the divisors $E$, $D$, $D'$, respectively to $C_i$.
	\item There exists a sub-limit linear series $l'$ of $l$ of type $\grdop{r-e+f}{d-e}$ such that each aspect $l'_{C_i}$ has a section vanishing on $D'_i$, for $i=1,2$.
\end{enumerate}
We denote space of such refined de Jonqui\`eres divisors by $DJ_{e,f}(X,l)$.
\end{defi}
The sections corresponding to the aspects $l_{C_i}$ and $l'_{C_i}$ will also vanish at the node $p$ of $X$ on each component $C_i$ in such a way that equality in (\ref{eq:vanishinglls}) is satisfied.  Using this we can give an alternative description of refined de Jonqui\`eres divisors as in Definition \ref{def:refineddj} that is more useful in practice.  Thus we call $E=D+D'$ a refined de Jonqui\`eres divisor of type $\Delta
_{e,f}^{r,d}$ for a refined limit linear series $l$ of type $\grd$ on $X=C_1\cup_p C_2$ if:
\begin{enumerate}[label=\alph*)]
	\item for each $C_i$, the aspect $l_{C_i}$ admits the refined de Jonqui\`eres divisors of type $\Delta_{e,f}^{r,d}$ given by 
	\[ D_i + D'_i + (d - \deg D_i - \deg D'_i)p, \]
	\item where $D'_i + (d-e-\deg D'_i)p$ belongs to a subseries of type $\grdop{r-e+f}{d-e}$ of $l_{C_i}$, and
	\item $D_i + (e-\deg D_i)p \in (C_i)_e$.
\end{enumerate}

Using the Osserman approach to limit linear series, we construct a representable functor whose corresponding scheme parametrises refined de Jonqui\`eres divisors for a family $\mathscr{X}\rightarrow B$ of curves of genus $g$ of compact type.

\begin{prop}
Let $\mathscr{X}\rightarrow B$ be a smoothing family of curves of compact type with a linear series $\ell$ of type $\grd$. There exists a scheme $\mathcal{DJ}_{e,f}(\mathscr{X}/B,\ell)\rightarrow B$ proper over $B$, compatible with base change, whose point over every $t\in B $ parametrises tuples $[\mathscr{X}_t, \mathscr{D}_t,\mathscr{D}'_t]$ of curves and divisors 
Moreover, every irreducible component of $\mathcal{DJ}_{e,f}(\mathscr{X}/B,\ell)$ has dimension at least $\dim B + e-f(r+1-e+f)-(s-1)(r-e+f)$.
\end{prop}
\begin{proof}
Step 1). To begin with, we construct a subfunctor $\mathscr{X}_{d-e,\ell'}$ of the functor of points of the symmetric product $Sym^{d-e}(\mathscr{X}/B)$ parametrising those divisors $\mathscr{D}'$ belonging to a linear series $\ell'$ of type $\grdop{r-e+f}{d-e}$ on $\mathscr{X}$.  We use a degeneracy locus construction from which it follows that it is representable by a scheme that is proper over $B$ and compatible with base change.

Let $T$ be a $B$-scheme.  Suppose that all the fibres of the family are nonsingular.  Thus $\ell'$ of type $\grdop{r-e+f}{d-e}$ on $\mathscr{X}/B$ is given by a pair $(\mathscr{L}',\mathscr{V}')$ as described at the end of Section \ref{sec:llsreview}.  Denote by $\mathcal{U}\subset\mathscr{X}\times_B Sym^{d-e}(\mathscr{X}/B)$ the universal family and let $\mathcal{U}_T=\mathcal{U}\times_B T$.  Consider the diagram below
\begin{figure}[H]\centering
  \begin{tikzpicture}
    \matrix (m) [matrix of math nodes,row sep=2em,column sep=1em,minimum width=1em]
  {
     & (\mathscr{X}\times_B T)\times_B Sym^{d-e}(\mathscr{X}/B) & \supset \mathcal{U}_T \\
     \mathscr{X}\times_B T & & T\\};
  \path[-stealth]
    (m-1-2) edge node [auto,swap] {$\tau_1$} (m-2-1)
            edge node [auto]{$\tau_2$}  (m-2-3);
 \end{tikzpicture}
 \end{figure}
\noindent where $\tau_1$, $\tau_2$ are the usual projections.  We say that the $T$-valued point $[\mathscr{X},\mathscr{D}']$ belongs to $\mathscr{X}_{d-e,\ell'}(T)$ if the $(r-e+f)$-th degeneracy locus of the map $\mathscr{V}'\rightarrow\tau_{2*}(\tau_1^*\mathscr{L}'\otimes\oo_{\mathcal{U}_T})$ is the whole of $T$.  By construction, $\mathscr{X}_{d-e,\ell'}$ is compatible with base change and has the structure of a closed subscheme, hence it is a functor represented by a proper scheme.  

Next, suppose that the family $\mathscr{X}$ has nodes that are not smoothed and denote by $\Gamma$ the dual graph of the fibres.  Then, from Section \ref{sec:llsreview} we have that $\ell'$ is a tuple $(\mathscr{L}'^v,\mathscr{V}'^v)_{v\in V(\Gamma)}$, where $\mathscr{L}'$ is a line bundle of degree $d$ on $\mathscr{Y}^v\times_B T$, for all components $\mathscr{Y}^v$ of $\mathscr{X}$ and $\mathscr{V}'^v$ is a subbundle of $\mathscr{L}'^v$ of rank $r+1$ on $T$ satisfying certain conditions discussed above.  Let $\mathscr{X}'(\mathscr{Y}^v,\mathscr{V}'^v)\subset Sym^{d-e}(\mathscr{Y}^v /B)$ be defined as follows: a $T$-valued point $[\mathscr{Y}^v,\mathscr{D}'^v]$ belongs to $\mathscr{X}'(\mathscr{Y}^v,\mathscr{V}'^v)$ if the $(r-e+f)$-th degeneracy locus of the map $\mathscr{V}'^v\rightarrow \pi_{2*}(\tau_1^* \mathscr{L}'^v \otimes \oo_{\mathcal{U}_T})$ is the whole of $T$, where $\tau_1$, $\tau_2$, and $\mathcal{U}_T$ are defined for $\mathscr{Y}^v$ in the same way as for the case of only smooth fibres. 
Now let $\mathscr{D}'^v$ denote the specialisation of the divisor $\mathscr{D}'$ on $\mathscr{X}$ to the component $\mathscr{Y}^v$ and $q_i^v$ the preimages of the nodes belonging to $\mathscr{Y}^v$.
 Therefore, we say that a $T$-valued point $[\mathscr{X},\mathscr{D}']$ belongs to $\mathscr{X}_{d-e,\ell'}(T)$ if, for all vertices $v\in\Gamma$, the $T$-valued points $[\mathscr{Y}^v,\mathscr{D}'^v,q_i^v]$ belong to $\mathscr{X}'(\mathscr{Y}^v,\mathscr{V}'^v)$, where the $q_i^v$ may appear with multiplicity.  

Finally, let $\mathscr{X}\rightarrow B$ be a smoothing family.  As previously discussed, let $\Gamma_0$ be the dual graph of the maximally degenerate fibre of $\mathscr{X}$ and fix a $v_0\in V(\Gamma_0)$.  Let $\mathscr{L}'$ be a line bundle on $\mathscr{X}\times_B T$ of degree $d-e$ on the component corresponding to $v_0$ and 0 otherwise and $\mathscr{V}'^v$ subbundles of rank $r+1$ of the twists $\pi_{2*}\mathscr{L}'^{\vec{d}^v}$ satisfying the conditions mentioned in Section \ref{sec:llsreview}.
Hence we know from before that $\ell'$ is a tuple $(\mathscr{L}',(\mathscr{V}'^v)_{v\in V(\Gamma_0)})$.  Let $b_0\in B$ be the point corresponding to the maximally degenerate fibre and, for a $T$-valued point  $[\mathscr{X},\mathscr{D}']\in Sym^{d-e}(\mathscr{X}/B)(T)$, denote by $v_i$ the vertices of $\Gamma_0$ corresponding to the components of $\mathscr{X}_{b_0}$ which contain points of the support of $\mathscr{D}'$.  We call the corresponding components $\mathscr{Y}^{v_i}$.  We define the $T$-valued points of $\mathscr{X}_{d-e,\ell'}$ in the same way as for the non-smoothing case, only now we need the degeneracy locus of a different bundle map, namely  $\mathscr{V}'^{v_i}|_{\mathscr{Y}^{v_i}}\rightarrow \tau_{2*}(\tau_1^*\mathscr{L}'^{\vec{d}^{v_i}}|_{\mathscr{Y}^{v_i}}\otimes\oo_{\mathcal{U}_T})$.

Step 2). Consider now the following:
\[ \{ ([\mathscr{X},\mathscr{D}'], \ell') \mid [\mathscr{X},\mathscr{D'}]\in \mathscr{X}_{d-e,\ell'}  \}\subset Sym^{d-e}(\mathscr{X}/B)\times\mathscr{G}^{r-e+f}_{d-e}(\mathscr{X}/B). \]
Using the valuative criterion one immediately sees that the space above is closed.  Projecting onto the first factor yields therefore a subfunctor $\mathscr{X}^{r-e+f}_{d-e}$ of the functor of points $Sym^{d-e}(\mathscr{X}/B)$ parametrising relative divisors of degree $d-e$ belonging to (limit) linear series of type $\grdop{r-e+f}{d-e}$. 

Step 3).  We now repeat the construction in Step 1), where we replace $\ell'$ of type $\grdop{r-e+f}{d-e}$ with $\ell$ of type $\grd$ and $Sym^{d-e}(\mathscr{X}/B)$ by $Sym^e(\mathscr{X}/B)\times_{\mathscr{X}}\mathscr{X}^{r-e+f}_{d-e}$.
The desired functor $\mathcal{DJ}_{e,f}(\mathscr{X}/B,\ell)$ is then constructed analogously to $\mathscr{X}_{d-e,\ell'}$ which endows it with a degeneracy locus structure and yields also the dimension estimate.
\end{proof}

\begin{rem}\label{rem:3}
As mentioned before (see Remark \ref{rem:2}), our proof of Theorem \ref{thm:refineddjdim} requires the use refined de Jonqui\`eres divisors in $DJ_{e,f}(C,l)$ where the corresponding linear series $l'$ has specified ramification at least $\alpha$ at a general point of $C$.  To take that into account in our construction above, one may simply consider $\ell'$ in Step 1) as an element of $\mathscr{G}^{r-e+f}_{d-e}(\mathscr{X}/B,\{q,\alpha\})$ parametrising (limit) linear series with prescribed ramification at a section $q$ of $\mathscr{X}\rightarrow B$ (see \autocite[Definition 4.5]{Os}). 
\end{rem}

\begin{rem}\label{rem:1}
Consider the forgetful map $\phi:\mathcal{DJ}_{e,f}(\mathscr{X}/B,\ell)\rightarrow\mathscr{X}$ and assume $\mathscr{X}\rightarrow B$ is a smoothing family with central fibre given by the nodal curve $X=C_1\cup_p C_2$ from above.  By base change, this forgetful map is a projective map.  The fibre of $\phi$ over a curve $\mathscr{X}_t$, with $t\in B$, is precisely $DJ_{e,f}(\mathscr{X}_t,\ell_t)$, i.e.~the space of refined de Jonqui\`eres divisors on a usual linear series if $\mathscr{X}_t$ is smooth and the space from Definition \ref{def:refineddj} on the central fibre.  
\end{rem}

\subsection{Proof of Theorem \ref{thm:refineddjdim}}\label{sec:proofrefineddjdim}
The proof of Theorem \ref{thm:refineddjdim} works in a similar way to the proof of the dimension theorem for usual de Jonqui\`eres divisors (Theorem 1.1 in \autocite{Un}). 

\subsubsection{Proof of Theorem \ref{thm:refineddjdim} part 1.}
The argument is based on an induction on the index of speciality $s$ of the linear series $l$ while keeping the Brill-Noether number $\rho(g,r,d)$ equal to 0.  The base case is the statement of Theorem \ref{thm:refineddjdim} for the canonical linear series $K_C$, which is the unique linear system with $s=1$ and vanishing Brill-Noether number.  The fact that this dimension result holds for $K_C$ follows from the discussion in Section \ref{sec:transvers}.  In fact, we shall need a more refined version of the statement of Theorem \ref{thm:refineddjdim}, namely the one given in Remark \ref{rem:ramif} that also takes into account the imposed ramification at a point of the series $l'$ from Definition \ref{def:refineddj}.

We now describe the induction step.  We construct aspect by aspect a refined limit linear series $l$ of type $\grd$ and with index of speciality $s=g-d+r$ on a nodal curve of compact type $X=C_1\cup_{p}C_2$ of genus $g$, where $C_1$ and $C_2$ are smooth, irreducible curves of genus $g_1$, $g_2$, with $g=g_1+g_2$ and such that $(C_1,p)$ and $(C_2,p)$ are general pointed curves.
Set $g_1=(s-1)(r+1)$ and $g_2=r+1$.
  On $C_1$ we take the aspect $l_{C_1}=l_1(rp)$, where $l_1\in G^{r}_{d-r}(C_1)$ and on $C_2$ the aspect $l_{C_2}=l_2((d-2r)p)$ where $l_2=G^{r}_{2r}(C_2)$ is the canonical bundle.  Note that with this choice $s(l_1)=s-1$, $s(l_2)=1$, and $\rho(g_1,r,d-r)=0$ hence the induction works by increasing the index of speciality $s-1\mapsto s$ while keeping the Brill-Noether number fixed and equal to 0.
 The induction hypothesis, is a refined version of the statement of Theorem \ref{thm:refineddjdim} for $C_i$ equipped with $l_i$, for $i=1,2$, namely 

\begin{hypo}
Under the hypotheses of Theorem \ref{thm:refineddjdim} a), let $q\in C$ be a general point.  The dimension of the space of refined de Jonqui\`eres divisors of type $DJ_{e,f}(C,l)$ such that the corresponding linear series $l'$ of type $\grdop{r-e+f}{d-e}$ has ramification sequence at least $\alpha$ at $q$ is $e-f(r+1-e+f)-(s-1)(r-e+f)-\sum_{i=0}^{r-e+f}\alpha_i$.
\end{hypo}

The limit linear series $l$ we constructed is indeed smoothable (see \autocite[Section 4.1]{Un} for the check that (\ref{eq:existencells}) is satisfied).
Thus there exists a family of curves $\pi:\mathscr{X}\rightarrow B$ with smooth generic fibre $\mathscr{X}_{\eta}$ and central fibre $\mathscr{X}_0=X$ equipped with a linear series $(\mathscr{L}_{\eta},\mathscr{V}_{\eta})$ of type $\grd$ whose limit on the central fibre is the limit linear series $l$ constructed above.  As explained in Section \ref{sec:llsreview}, we therefore have a linear series $\ell$ on $\mathscr{X}\rightarrow B$ specialising to $l$ at $0\in B$ and to a usual $\grd$ at $\eta\in B$.
Furthermore, the space of refined de Jonqui\`eres divisors $\mathcal{DJ}_{e,f}(\mathscr{X}/B,\ell)$ restricts to $DJ_{e,f}(X,l)$ at $0\in B$.  The strategy is to show that, within the parameter bounds of the statement of Theorem \ref{thm:refineddjdim} a), $\dim DJ_{e,f}(X,l)\leq e-f(r+1-e+f)-(s-1)(r-e+f)$.  From the upper semi-continuity of fibre dimension applied to the map $\phi$ from Remark \ref{rem:1}, it follows that $\dim DJ_{e,f}(\mathscr{X}_{\eta},\ell_{\eta})\leq e-f(r+1-e+f)-(s-1)(r-e+f)$ for a smoothing of $X$ to a general curve $\mathscr{X}_{\eta}$ equipped with a general linear series $\ell_{\eta}$.  This combined with the upper bound in (\ref{eq:expdim}) yields the statement of Theorem \ref{thm:refineddjdim} a).
 
From the discussion in Section \ref{sec:degeneration} we see that if the limit of a divisor $E=D+D'$ in $\Delta_{e,f}^{r,d}$ is a refined de Jonqui\`eres divisor for the limit linear series $l$ then the following hold true:
\begin{enumerate}[label=\arabic*), wide, labelwidth=!, labelindent=0pt]
	\item On $C_1$ we have a refined de Jonqui\`eres divisor for $l_{C_1}$ of the form
		\[D_1 + D'_1 + (d-\deg D_1 - \deg D'_1)p \in l_{C_1},\]
where
\begin{align*}
D'_1 + (d-e-\deg D'_1)p &\in l'_{C_1}=g^{r-e+f}_{d-e},\\
D_1 + (e-\deg D_1)p &\in C_e.
\end{align*} 
	Recall that $l'$ is a sub-limit linear series of $l$, which in turn means that the aspect $l'_{C_1}$ is a sub-linear series of $l_{C_1}$.  Moreover, by construction we know that the ramification sequence of $l_{C_1}$ at $p$ is $(r,\ldots,r)$.  Thus $\alpha(l'_{C_1},p)\geq (r,\ldots,r)$, since the vanishing sequence of $l'_{C_1}$ is a subsequence of that of $l_{C_1}$.  Therefore $l'_{C_1}-rp\in G^{r-e+f}_{d-e-r}(C_1)$ and we conclude that we have a refined de Jonqui\`eres divisor on $C_1$ for the linear series $l_1$ as well, namely:				\[ D_1 + D'_1 + (d-r-\deg D_1 - \deg D'_1)p \in l_1, \]
where 
\begin{align}
D'_1 + (d-r-e-\deg D'_1)p &\in l'_1=g^{r-e+f}_{d-r-e},\label{eq:div1}\\
D_1 + (e-\deg D_1)p &\in C_e.\nonumber
\end{align} 
Thus from the lower bound on the dimension in (\ref{eq:expdim}) and the induction hypothesis we see that the dimension of the space of such refined de Jonqui\`eres divisors is 
\[\dim DJ_{e,f}(C_1,l_1) = e-f(r+1-e+f) - (s-2)(r-e+f).\]
	\item Similarly, on $C_2$ we have a refined de Jonqui\`eres divisor for $l_2$ of the form:
	\[ D_2 + D'_2 + (2r-\deg D_2 - \deg D'_2)p \in l_2, \]
where 
\begin{align}
D'_2 + (2r-e-\deg D'_2)p &\in l'_2=\grdop{r-e+f}{2r-e},\label{eq:div2}\\
D_2 + (e-\deg D_2)p &\in C_e.\nonumber
\end{align}
From Lemma \ref{lemma:trans} for the canonical linear series we obtain the dimension of the space of refined de Jonqui\`eres divisors corresponding to $l_2$ to be
\[ \dim DJ_{e,f}(C_2,l_2) = e-f(r+1-e+f). \]	 
\end{enumerate}

The idea is now to give an upper bound to the dimension of $DJ_{e,f}(X,l)$ by using the two spaces of refined de Jonqui\`eres divisors corresponding to $l_1$ and $l_2$. 
We observe first that
\begin{align*}
\dim DJ_{e,f}(X,l)&\leq \dim DJ_{e,f}(C_1,l_1)+ \dim DJ_{e,f}(C_2,l_2)\\
&=2(e-f(r+1-e+f)) - (s-2)(r-e+f).
\end{align*}
 In order to obtain a more accurate bound one must take into account the fact that the glueing point $p$ is in fact a base point of the linear series $l'_1=\grdop{r-e+f}{d-r-e}$ and $l'_2=\grdop{r-e+f}{2r-e}$ on $C_1$ and $C_2$, respectively.  To see this, we now show that the coefficient of $p$ in both (\ref{eq:div1}) and (\ref{eq:div2}) is always strictly positive.  Let $d_i=\deg D_i$ and $d'_i=\deg D'_i$, for $i=1,2$.  Assume first that 
\begin{equation}\label{eq:div3}
d-r-e-d'_1\leq 0\text{ and }2r-e-d'_2 \leq 0. 
\end{equation}   
Combining this with the fact that $d_1+d_2=e$ and $d'_1+d'_2=d-e$ yields
\begin{align*}
&d-r-d_1-d'_1 \leq d_2,\\
&2r-d_2-d'_2\leq d_1.
\end{align*}
Adding up the two above inequalities gives $d+r\leq d'_1 + d'_2 = d-e$, which cannot hold.  Hence at least one of the coefficients of $p$ in (\ref{eq:div1}) and (\ref{eq:div2}) must be strictly positive.  So assume now that $d'_1=d-r-e-x$, for some strictly positive $x$ and that the second  inequality in (\ref{eq:div3}) remains unchanged.  We then have that 
\begin{align*}
&d-r-d_1-d'_1 = d_2+x,\\
&2r-d_2-d'_2\leq d_1,
\end{align*}
which added up yield $x\geq e+r$.  Hence  
\[ d'_1\leq d-r-e-e-r=d-2r-2e. \]
In addition, since $d_1\leq e$, we obtain
\begin{equation}\label{eq:div4}
d_1+d'_1\leq d-2r-e.
\end{equation}
However, note that if all the points in the supports of $D$ and $D'$ specialise on only one component $C_i$, then this would contradict the imposed ramification condition at $p$ on the aspect $l_{C_i}$.  Indeed, we must have
\begin{align*}
d-2r \leq &d_1 + d'_1 \leq d-r,\\
r\leq &d_2 + d'_2 \leq 2r.
\end{align*}
But this gives rise to a contradiction with the inequality in (\ref{eq:div4}) and we conclude that the coefficient of $p$ in (\ref{eq:div2}) must also be strictly positive.  A similar computation gives the same conclusion if we start with the assumption that $d'_2<2r-e$ and $d-e-r-d'_1\leq 0$.

Thus $p$ is a base point of the two series $l'_1$ and $l'_2$ on $C_1$ and $C_2$, respectively, so it follows that both $\alpha(l'_1,p)$ and $\alpha(l'_2,p)$ are at least $(1,1,\ldots,1)$.  This together with the induction hypothesis yield the following bound for the dimension of $DJ_{e,f}(X,l)$:
\begin{align}
\dim DJ_{e,f}(X,l)&\leq e-(r+1-e+f)-(s-2)(r-e+f)-\sum_{i=0}^{r-e+f}\alpha_i(l'_1,p)\nonumber \\
&+ e-f(r+1-e+f)-\sum_{i=0}^{r-e+f}\alpha_i(l'_2,p)\nonumber\\ 
&\leq 2(e-f(r+1-e+f)) - s(r-e+f)-2.\label{eq:bound1}
\end{align}
Hence, if $e-f(r+1-e+f)\leq r-e+f+2$, then from (\ref{eq:bound1}) we get
\[\dim DJ_{e,f}(X,l)\leq e-f(r+1-e+f)-(s-1)(r-e+f).\]
In addition, for a general point $q$ belonging to any of the components $C_i$ at which the aspect $l'_i$ has ramification sequence at least $\alpha$, we have that
 \[\dim DJ_{e,f,\alpha}(X,l)\leq e-f(r+1-e+f)-(s-1)(r-e+f)-\sum_{i=0}^{r-e+f}\alpha_i.\]
Thus, using upper semi-continuity of fibre dimension and the lower bound for the dimension of refined de Jonqui\`eres divisors with imposed ramification at a point $q$ from Remark \ref{rem:2} we see that $DJ_{e,f,\alpha}(C,l)$ is indeed of expected dimension for a general curve $(C,q)$.  This concludes the induction step.

\subsubsection{Proof of Theorem \ref{thm:refineddjdim} part 2.}
The proof for the second part of the theorem is also by induction and follows the same idea as the for Theorem \ref{thm:refineddjdim} part 1.  The only difference is that this time, the induction step increases the Brill-Noether number $\rho(g,r,d)$ while keeping the index of speciality $s$ fixed.  The base case is the statement of Theorem \ref{thm:refineddjdim} 1. and Lemma \ref{lemma:trans}.

The induction step works as follows:  as before, we construct a refined linear series $l$ of type $\grd$ on a nodal curve of compact type $X$ aspect by aspect.  We take $X=C_1\cup_p C_2$ of genus $g$, where $(C_1,p)$ and $(C_2,p)$ are general pointed curves of genus $g_1=g-1$ and $g_2=1$, respectively.  On $C_1$ consider the aspect $l_{C_1}=l_1(p)$ with $l_1\in G^{r}_{d-1}(C_1)$ and on $C_2$ the aspect $l_{C_2}=l_2((d-r-1)p)$ with $l_2\in G^r_{r+1}(C_2)$.  Note that now $s(l_1)=s$ while $\rho(g-1,r,d-1)=\rho(g,r,d)-1$.  Thus the induction increases the Brill-Noether number $\rho(g,r,d)-1\mapsto\rho(g,r,d)$ while keeping the index of speciality $s$ fixed (this follows because $s(l_1)=(g-1)-(d-1)+r=s=s(l)$).  The induction hypothesis is analogous to the one from the proof of Theorem \ref{thm:refineddjdim} part 1.

As before, one can easily check that the limit linear series $l$ is smoothable (see \autocite[Section 4.2]{Un}) so that there exists a family of curves $\mathscr{X}\rightarrow B$ with smooth generic fibre $\mathscr{X}_{\eta}$ and central fibre given by our $X$ and equipped with a linear series $(\mathscr{L}_{\eta},\mathscr{V}_{\eta})$ whose limit on $X$ is $l$ constructed above.  The strategy is to again show that $DJ_{e,f}(X,l)\leq e-f(r+1-e+f)-(s-1)(r-e+f)$ and conclude via upper semi-continuity that $\dim DJ_{e,f}(\mathscr{X}_{\eta},\ell_{\eta})\leq e-f(r+1-e+f)-(s-1)(r-e+f)$ for a smoothing of $X$ to a general curve $\mathscr{X}_{\eta}$ with a general $\ell G^r_d(\mathscr{X}_{\eta})$.  The statement of Theorem \ref{thm:refineddjdim} part 2. then follows.

As before, we have that if the limit of $E=D+D'\in\Delta^{r,d}_{e,f}$ is a refined de Jonqui\`eres divisor for $l$ if the following are satisfied:

\begin{enumerate}[label=\arabic*), wide, labelwidth=!, labelindent=0pt]
	\item On $C_1$ we have a refined de Jonqui\`eres divisor for $l_{C_1}$ of the form
	\[ D_1 + D'_1 + (d-\deg D_1-\deg D'_1)p\in l_{C_1}, \]
	where  
	\begin{align*}
	D'_1 + (d-e-\deg D'_1)p &\in l'_{C_1}=\grdop{r-e+f}{d-e},\\
	D_1 +(e-\deg D_1)p &\in C_e.
	\end{align*}
	Since $l'_{C_1}$ is a sub-linear series of $l_{C_1}$, we immediately have that $\alpha(l'_{C_1},p)\geq (1,1,\ldots,1)$.  Thus $l'_{C_1}-p\in G^{r-e+f}_{d-1-e}(C_1)$ and we therefore have a refined de Jonqui\`eres divisor on $C_1$ for $l_1$: 
	\[	D'_1 + (d-1-e-\deg D_1 - \deg D'_1)p\in l_1, \]
	where
	\begin{align*}
	D'_1 + (d-1-e-\deg D'_1)p &\in l'_1 = \grdop{r-e+f}{d-1-e},\\
	D_1 + (e-\deg D_1)p &\in C_e.
	\end{align*}
	Hence, using the induction hypothesis, we obtain:
	\[\dim DJ_{e,f}(C_1,l_1)=e-f(r+1-e+f)-(s-1)(r-e+f).\]
	\item Analogously, on $C_2$ we have the following refined de Jonqui\`eres divisor for $l_2$:
	\[ D_2 + D'_2 + (d-r-1-\deg D_2 - \deg D'_2 )p\in l_2, \]
	where
	\begin{align*}
	D'_2 + (r+1-e-\deg D'_2)p &\in l'_2=\grdop{r-e+f}{r+1-e}\\
	D_2 + (e-\deg D_2)p\in C_e.
	\end{align*}
	Lemma \ref{lemma:trans} for non-special linear series then yields
	\[ \dim DJ_{e,f}(C_2,l_2)=e-f(r+1-e+f)+r-e+f. \]
\end{enumerate}
Arguing like for the previous case, one sees that $p$ is a base point of the linear series $l'_1$ and $l'_2$ and this implies that both $\alpha(l'_1,p)$ and $\alpha(l'_2,p)$ are at least $(1,1,\ldots,1)$.  Thus, from the induction hypothesis we obtain the same bound for the dimension of $DJ_{e,f}(X,l)$ as before, namely:
\[ \dim DJ_{e,f}(X,l)\leq 2(e-f(r+1-e+f))-s(r-e+f)-2. \]
We then argue just like in the previous case to conclude the induction step.

\printbibliography{}
\end{document}